\documentclass[sn-mathphys-num]{sn-jnl}% Math and Physical Sciences Numbered Reference Style
\usepackage{lmodern}
\usepackage{nicefrac}
\usepackage{graphicx}%
\usepackage{multirow}%
\usepackage{amsmath,amssymb,amsfonts}%
\usepackage{amsthm}%
\usepackage[title]{appendix}%
\usepackage{xcolor}%
\usepackage{textcomp}%
\usepackage{manyfoot}%
\usepackage{booktabs}%
\usepackage{algorithm}%
\usepackage{algorithmicx}%
\usepackage{algpseudocode}%
\usepackage{listings}%
\usepackage{url}
\newtheorem{theorem}{Theorem}
\newtheorem{proposition}[theorem]{Proposition}
\newtheorem{remark}{Remark}%
\newtheorem{lemma}{Lemma}%
\newtheorem{corollary}{Corollary}%
\begin{document}

\title[Evaluation codes from linear systems of conics]{Evaluation codes from linear systems of conics}


\author[1,2]{\fnm{Barbara} \sur{Gatti}}\email{barbara.gatti@unisalento.it}
\equalcont{These authors contributed equally to this work.}

\author[2]{\fnm{G\'abor} \sur{Korchm\'aros}}\email{gabor.korchmaros@unibas.it}
\equalcont{These authors contributed equally to this work.}

\author[1,2]{\fnm{Gioia} \sur{Schulte}}\email{gioia.schulte@unisalento.it}
\equalcont{These authors contributed equally to this work.}

\affil[1]{\orgdiv{Department of Mathematics and Physics ”Ennio de Giorgi”}, \orgname{University of Salento}, \orgaddress{\street{Via per Arnesano}, \city{Lecce}, \postcode{73100 }, \state{Italy}}}

\affil[2]{\orgdiv{Department of Mathematics, Computer Science and Economics}, \orgname{University of Basilicata}, \orgaddress{\street{Contrada Macchia Romana}, \city{Potenza}, \postcode{85100}, \state{Italy}}}


\abstract{The Datta-Johnsen code is an evaluation code where the linear combinations of elementary symmetric polynomials are evaluated on the set of  all points with pairwise distinct coordinates in an affine space of dimension $\ge 2$ over a finite field $\mathbb{F}_q$. A generalization is obtained by taking a low dimensional linear system of symmetric polynomials. The odd characteristic case was the subject of a recent paper. Here, the even characteristic case is investigated.
}





\maketitle

\section{Introduction}\label{sec1}
Finite geometry is a major source of construction of linear codes, in particular of evaluation codes where the codewords are constructed from polynomials of $\mathbb{F}_q[X_1,\ldots,X_m]$ by evaluating them on some set of points of the $m$-dimensional affine space $AG(m,\mathbb{F}_q)$  over the finite field $\mathbb{F}_q$ of order $q$. Well known evaluation codes are the Reed-Solomon codes, Reed-Muller codes, monomial codes, Cartesian codes, and toric codes. The constructions of evaluation codes are quite diverse, the unique constraint being that the chosen polynomials must form a linear system, that is, a finite dimensional $\mathbb{F}_q$-subspace of $\mathbb{F}_q[X_1,\ldots,X_m]$.  This occurs, for instance, when all symmetric polynomials of bounded degree in $\mathbb{F}_q[X_1,\ldots,X_m]$ are taken. Datta and Johnsen \cite{datta2023codes} investigated the case where the linear system comprises all $\mathbb{F}_q$-linear combinations of elementary symmetric polynomials while the evaluation set consists of all distinguished points in $AG(m,\mathbb{F}_q)$, i.e. on all points with pairwise distinct coordinates. The \emph{Datta-Johnsen code} is a $$\Big[\binom{q}{m}m!,\,m+1,\,(q-m)\binom{q-1}{m-1}(m-1)!\Big]_q$$ linear code with relative minimum distance equal to that of the Reed-Muller code. To improve the relative dimension (or rate) of their code, Datta and Jonhsen introduced a modified version in the same paper, named later \emph{reduced Datta-Johnsen code}, that has the same relative minimum distance, but a better rate. Their idea was to keep the linear system of symmetric polynomials but to ``reduce'' the set of distinguished points of $AG(m,\mathbb{F}_q)$  taking the representatives of equivalence classes of distinguished points of $AG(m,\mathbb{F}_q)$, where two distinct distinguished points are equivalent when they have the same coordinates but in a different order. For $m < q$, the reduced Datta-Johnsen code $C_m'$ is a non-degenerate $$\Big[\binom{q}{m},\,m+1,\,\binom{q}{m}- \binom{q-1}{m-1}\Big]_q$$ linear code.
A variation of the Datta-Johnsen codes was introduced and investigated in \cite{Mi2} with an approach combining Galois theoretical methods with Weil-type bounds for hypersurfaces.

The reduced Datta-Johnsen code was generalized in \cite{GKNPS} for  any linear system $V$ consisting of symmetric polynomials in $\mathbb{F}_q[X_1,\ldots,X_m]$. The {\emph{reduced generalized Datta-Johnsen code}} is the evaluation code where the evaluating polynomials are those in $V$ and the evaluation set $\mathcal{Q}$ is the representatives of equivalence classes of distinguished points of $AG(m,\mathbb{F}_q)$. Let  $\Phi_m$ denote the map taking any polynomial $f\in\mathbb{F}_q[X_1,\ldots,X_m]$ to the polynomial $\Phi_m(f)=f(\sigma_m^1(x),\dots,\sigma_m^m(x))$  where $\sigma_m^i(x)$ denotes the $i$-th elementary symmetric polynomial. The fundamental theorem of symmetric polynomials states that
${\rm{Im}}(\Phi_m)$ consists of all symmetric polynomials in $\mathbb{F}_q[X_1,\ldots,X_m]$.
Since $\Phi_m$ is an $\mathbb{F}_q$-linear map, every $\mathbb{F}_q$-subspace of symmetric polynomials is the image of a (unique) $\mathbb{F}_q$-subspace of $\mathbb{F}_q[X_1,\ldots,X_m]$. Therefore,
 the linear system of all polynomials $\Sigma_{m,t}$ of degree $\le t$, as well as any of its linear subsystems $\Sigma_{m,t}(r)$ of dimension $r$  defines an $\mathbb{F}_q$-subspace of symmetric polynomials. Here, $\Sigma_{m,1}$ corresponds to the subspace generated by the elementary symmetric polynomials in $m$ indeterminates. Thus, the Datta-Johnsen code and its reduction correspond to the simplest case, i.e. $\Sigma_{m,1}$. In \cite{GKNPS} it was pointed out that $\mathcal{Q}$ can be identified by the set of the unramified points of the quotient variety $\mathbb{F}_q^m/{\rm{Sym}}_m$, so that the associated generalized reduced Datta-Johnsen code turns out to be equivalent to the evaluation code of the chosen linear (sub)system where the polynomials are evaluated on the set $\Delta$ of unramified points of the quotient variety $\mathbb{F}_q^m/{\rm{Sym}}_m$. In \cite{GKNPS}, an embedding of $\Delta$ in $AG(m,\mathbb{F}_q)$ was described allowing to interpret the codewords in terms of intersections of $\Delta$ with certain hypersurfaces. This embedding  has given a motivation to consider the connections between, on one side, the fundamental parameters of the generalized Datta-Johnsen codes, especially their weight distributions, and, on the other side, certain enumerative questions concerning intersections of relevant objects in Finite geometry and Algebraic geometry over finite fields. The study of the finite and algebraic geometry counterpart is itself of interest, and yet to be carried out for a general choice of $m$ and $t$.

 From now on, let $m=t=2$.

 The odd characteristic case was thoroughly worked out in \cite{GKNPS} where it was shown that $\Delta$ coincides with the set $\mathcal{E}$ of all external points to a given parabola in $AG(2,\mathbb{F}_q)$. Actually, this result was sufficient to compute the fundamental parameters of the generalized Datta-Johnsen code since the spectrum of the possible sizes of the pointsets cut out on $\mathcal{E}$ by conics had previously been determined; see \cite{AFKL}.
 An essential idea in \cite{AFKL} was to express those sizes in terms of the number of points of rational and elliptic curves over $\mathbb{F}_q$.

 The even characteristic case is the subject of this paper. Clearly, $\Delta$ cannot be the set of external points to a conic, as such points do not exist in $AG(2,\mathbb{F}_q)$ for $q$ even. Nevertheless, a kind of analogy with the odd characteristic case is possible. In fact, $\mathcal{E}$ can be regarded as the set covered by a certain family of parabolas. More precisely, if the parabola $\mathcal P$ whose external points form $\mathcal{E}$ is taken with the canonical equation $Y=X^2$, then the parabola $\mathcal{P}_c$ of equation $Y=X^2-c$ with $c\in \mathbb{F}_q$ entirely consists of external points to $\mathcal{P}$ if and only if $c$ is a non-zero square in $\mathbb{F}_q$, i.e. $c\in \square ^*$. Therefore, the parabolas $\mathcal{P}_c$ with $c$ ranging over $\square^*$ cover each point of $\Delta$ exactly one time. In Section \ref{dede} we show that if $q$ is even and $O=(0,0)$ is the origin, then $\Delta\cup \{O\}$ is covered by the parabolas $\mathcal{P}_a$ of equation $Y=aX^2$ with $a$ ranging over the elements of zero trace in $\mathbb{F}_q$ each point of $\Delta$ being covered exactly one time. In other words, if $q$ is even, then
 $$\Delta=\{P(x,ax^2)|x\in \mathbb{F}_q^*, \mathfrak{Tr}(a)=0\}.$$
 Unfortunately, this analogy is not close enough to obtain the spectrum for even $q$ by adapting the arguments used in \cite{AFKL}. Nevertheless, the possible sizes of the pointsets cut out on $\Delta$ by conics can be determined still relying on the number of points of rational and elliptic curves over $\mathbb{F}_q$. This is shown in Section \ref{dede} by a careful analysis of certain plane algebraic curves defined over $\mathbb{F}_q$. The main result in this direction is Corollary \ref{cor28112025} which states that
 if $C$ is a non-degenerate conic of $AG(2,\mathbb{F}_q)$ with equation
$$a_{11}X^2+a_{12}XY+a_{22}Y^2+a_{13}X+a_{23}Y+a_{33}=0,$$
then $$
|\Delta\cap C|\le \textstyle{\frac{1}{2}}(\sqrt{q}+1)^2-1.
$$
with one exception
 $$|\Delta\cap C|=q-1, a_{12}=a_{22}=0, \mbox{ and } a_{13}=a_{33}=0 \mbox{ or } a_{13}^2=a_{33}\ne 0.$$
For parabolas $C$ including the above exceptional case, Theorem \ref{pro19112025AA} provides the full spectrum of the sizes of $\Delta\cap C$.

Our results on the reduced generalized Datta-Johnsen are given in  Section \ref{ap}. The reduced generalized Datta-Johnsen code arising from the linear system of all conics is  $\left[\frac{1}{2} q(q-1),6,\frac{1}{2}q(q-3)\right]_q$.
In Section \ref{exa}, we work out two particular cases. The first one, see Construction 1, gives the following result.
\begin{theorem}
\label{main271224} For any power $q\ge 4$ of $2$, there exist reduced generalized Datta-Johnsen codes  $\left[\textstyle{\frac{1}{2}} q(q-1),3,d\right]_q$ whose minimum distance $d$ is at least $\textstyle{\frac{1}{2}}\left(q^2-2q-2\sqrt{q}+1\right)$. The weights of the non-zero codewords fall into
the interval
$$\left[\textstyle{\frac{1}{2}}(q-2\sqrt{q}-2),
 \textstyle{\frac{1}{2}}(q+2\sqrt{q}-1)\right].$$
\end{theorem}
For $q=8$, a Magma aided exhaustive search shows that either $d=21$, or $d=22$. This shows that the bound is sharp. Furthermore, the dual of some reduced generalized Datta-Johnsen code in Construction 1 is $[28,25,3]_8$ whose minimum distance is equal to the optimal value.

Construction 2 provides, for any power $q\ge 4$ of $2$, a reduced generalized Datta-Johnsen codes  $\left[\frac{1}{2} q(q-1),4,\frac{1}{2}q(q-3)\right]_q$  with weight distribution
$$\left\{\textstyle{\frac{1}{2}}q(q-3),\,\textstyle{\frac{1}{2}}(q^2-3q+2),\, \textstyle{\frac{1}{2}}q^2-q,\,\textstyle{\frac{1}{2}}q^2-q+1,\, \textstyle{\frac{1}{2}}q(q-1)\right\}.$$

\section{Background}
\subsection{Symmetric polynomials in two variables}
\label{smp}
Let $\mathbb{K}$ be any field.
A polynomial $F\left(X_1,X_2\right)\in \mathbb{K}\left[X_1,X_2\right]$ in the indeterminates $X_1,X_2$ with coefficients in $\mathbb{K}$ is \emph{symmetric} if $F\left(X_1,X_2\right)=F\left(X_2,X_1\right)$.
The \emph{elementary symmetric polynomials} are
$\sigma^1(X_1,X_2)=X_1+X_2$ and $\sigma^2(X_1,X_2)=X_1X_2$.

Let $f(Y)$ be any monic polynomial of degree $2$ in the unique indeterminate $Y$ with coefficients in $\mathbb{K}$. Let $y_1,y_2$ be the (not necessarily distinct) roots of $f(Y)$ in an algebraic closure of $\mathbb{K}$. Then
\begin{equation*}\label{monic}
    f(Y)=Y^2-\sigma_1(y_1,y_2)Y+\sigma_2(y_1,y_2)=Y^2-(y_1+y_2)+y_1y_2.
\end{equation*}

For a polynomial $F(X_1,X_2)\in \mathbb{K}[X_1,\ldots,X_m]$, substituting $X_i$ with the $i$-th symmetric polynomial provides a polynomial in $\mathbb{K}\left[X_1,X_2\right]$, namely
$$G\left(X_1,X_2\right)=F\left(X_1+X_2,X_1X_2\right)$$ which is symmetric. From the fundamental theorem on symmetric polynomials, every symmetric polynomial $G\in \mathbb{K}\left[X_1,X_2\right]$ arises in this way from a unique (not necessarily symmetric) polynomial $F\in \mathbb{K}\left[X_1,X_2\right]$. This defines  a vector space monomorphism $\Phi_2$ from $\mathbb{K}[X_1,X_2]$ onto its subspace $\mathbb{K}[X_1,X_2]^s$ comprising symmetric polynomials. Therefore, for any $\mathbb{K}$-subspace $\Sigma$ of $\mathbb{K}[X_1,X_2]$ (called linear system over $\mathbb{K}$), $\Phi_2(\Sigma)$ is a $\mathbb{K}$-subspace of symmetric polynomials, and the converse is true, as well.
\subsection{Plane algebraic curves over finite fields}
\label{av} For the theory of plane algebraic curves in positive characteristic, the reader is referred to \cite[Chapters 1-5]{HKT}.

From now on we assume $q\ge 4$ is a power of $2$.

Fix an algebraic closure $\mathbb{K}=\overline{\mathbb{F}}_q$, and let $AG(2, \mathbb{K})$ be the affine plane over $\mathbb{K}$ where an affine reference system $(X,Y)$ is fixed.  The point in $AG(2,\mathbb{K})$ with coordinates $x,y$ is denoted by $(x,y)$. A line of $AG(2,\mathbb{K})$ has equation either $Y=mX+b$, or $X=c$.
For a non-constant polynomial $F(X,Y)$ over $\mathbb{K}$, the (affine) plane curve $\mathcal F$ of equation $F(X,Y)=0$ is defined to be the set of zeros of $F(X,Y)$. If the coefficients of $F(X,Y)$ belong to $\mathbb{F}_q$, then $\mathcal{F}$ is a \emph{curve defined over} $\mathbb{F}_q$, and it is also called a curve of $PG(2,\mathbb{F}_q)$. To simplify notation, we sometimes identify $\mathcal F$ with its points in $AG(2,\mathbb{F}_q)$, and in such a case
\[\mathcal{F}=\{(x, y)\in AG(2,\mathbb{F}_q)\mid F(x, y) = 0\}.\]
The \textit{degree} of $\mathcal{F}$ is $\deg F$. A \textit{component} of $\mathcal F$ is any curve $\mathcal{G}$ of equation $G(X,Y)=0$ such that $G(X,Y)$ divides $F(X,Y)$. A curve $\mathcal F$ is \textit{irreducible} if $F$ is irreducible over $\mathbb{K}$; otherwise it is \textit{reducible} and splits into irreducible curves over $\mathbb{K}$, the components of $\mathcal{F}$.

A \emph{conic} is a plane curve of degree $2$. An irreducible conic of $AG(2,\mathbb{F}_q)$ is either a hyperbola, or a parabola, or an ellipse. A reducible conic (also called degenerate conic) of  $AG(2,\mathbb{F}_q)$ is either a line (counted twice), or two intersecting lines, or two parallel lines, or a line, or a single point; see \cite[Section 7.2]{HJWP}.  Accordingly, if the conic is defined over $\mathbb{F}_q$, its size equals $q-1,q,q+1,2q-1,2q,q,1$. An equation of a conic $C_2$ of $AG(2,\mathbb{F}_q)$ is
\begin{equation}
\label{eq21072025}
a_{11}X^2+a_{12}XY+a_{22}Y^2+a_{13}X+a_{23}Y+a_{33}=0
\end{equation}
with some non-vanishing coefficient and $C_2$ is degenerate if and only if
\begin{equation}
\label{conde}
a_{11}a_{23}^2+a_{12}a_{23}a_{13}+a_{22}a_{13}^2+a_{33}a_{12}^2=0.
\end{equation}
In particular, $C_2$ is degenerate when either $a_{11}=a_{12}=a_{22}=0$, or $a_{11}=a_{13}=a_{33}=0$, or $a_{11}=a_{22}=a_{23}=0$, or $a_{22}=a_{23}=a_{33}=0$. In the first case, $C_2$ coincides with the line of equation $a_{13}X+a_{23}Y+a_{33}=0$.
In  the second and forth cases $C_2$ splits into two lines both defined over $\mathbb{F}_q$. In the second case they have equations $y=0$ and $a_{12}X+a_{22}Y+a_{23}=0$ while in the forth case,  $X=0$ and $a_{11}X+a_{12}Y+a_{13}=0$. In the third case, $C_2$ has equation $a_{11}X^2+a_{12}X+a_{22}=0$ and $C_2$ splits into two lines which are defined either over $\mathbb{F}_q$, or over $\mathbb{F}_{q^2}$ according as $\mathfrak{Tr}(a_{11}a_{22}/a_{12}^2)$ is equal to $0$ or $1$;  apart from two exceptions, namely when $C_2$ coincides with a line where either  $a_{11}=0$, or $a_{12}=0$ and the line has equation $a_{12}X+a_{22}=0$, or $X+\sqrt{a_{22}/a_{11}}=0$.

A quadratic transformation $\omega$ of $AG(2,\mathbb{K})$ is a map $(X,Y)\rightarrow (U(X,Y),V(X,Y))$ where $U(X,Y),V(X,Y)$ are quadratic polynomials. Here, $\omega$ takes the point $(x,y)$ to the point $(U(x,y),V(x,y))$ and the curve $\mathcal{F}$ of equation $F(X,Y)$ to the curve $\mathcal{G}$ of equation $G(X,Y)=0$ such that  $G(U(X,Y),V(X,Y))=F(X,Y)$.

The projective closure $PG(2, \mathbb{K})$ of $AG(2,\mathbb{F}_q)$ is equipped with homogeneous coordinates $(X_1:X_2:X_3)$ with $X=X_1/X_3,Y=X_2/X_3$, so that $F$ is replaced by the corresponding homogeneous polynomial $F^*\in \mathbb{K}[X_1,X_2,X_3]$ of the same degree. The line of equation $X_3=0$ is the line at infinity.
The projective closure of the affine curve $\mathcal{F}$ is the projective plane curve consisting of all points $P$ whose coordinates satisfy the equation $F^*=0$, i.e.
\[\{P=(x_1:x_2:x_3) \in PG(2,\mathbb{K})\mid F^*(x_1,x_2,x_0)= 0\}.\]
To simplify notation, $\mathcal{F}$ will also denote the projective closure of $\mathcal{F}$.

A \emph{cubic curve} is a plane curve of degree $3$.
An irreducible cubic $C_3$ is either a \emph{rational} or an \emph{elliptic curve} according as $C_3$ has a singular point or does not. Here, a singular point of $C_3$ is a point $P$ in $PG(2,\mathbb{K})$ such that the intersection multiplicity $I(P,C\cap \ell)$ at $P$ between $C_3$ and any line $\ell$ of $PG(2,\mathbb{K})$ equals $2$ with one or two exceptions according as $P$ is a node or a cusp. The exceptional lines are the \emph{tangents} to $C_3$ at $P$.

If $C_3$ is an irreducible and rational cubic curve defined over $\mathbb{F}_q$ then there are three possibilities for the number $N_q$ of points of $C_3$ in $PG(2,\mathbb{F}_q)$ depending upon the behavior of the unique singular point of $C_3$; namely $N_q=q$ when $P$ is a node with both tangents defined over $\mathbb{F}_q$, $N_q=q+1$ when $P$ is a cusp, and $N_q=q+2$ when $P$ is an (isolated) node with both tangents defined over the quadratic extension of $\mathbb{F}_q$. In $AG(2,\mathbb{F}_q)$, we have some more possibilities according to the number of points of $C_3$ lying on the line at infinity. In fact, let $AG(2,\mathbb{F}_q)$ be the affine plane whose projective closure is $PG(2,\mathbb{F}_q)$. Suppose that $C_3$ has either $2$ or $3$ points at infinity defined over $\mathbb{F}_q$, none of them is singular. Then the number of $\mathbb{F}_q$-rational point of $C_3$ in $AG(2,\mathbb{F}_q)$ belongs to the interval $[q-3,q]$.    

If $C_3$ is an elliptic curve, then the points of $C_3$ in $PG(2,\mathbb{F}_q)$ are exactly the $\mathbb{F}_q$-rational point of $C_3$. Thus, their number belongs to the interval $[q+1-2\sqrt{q},q+1+2\sqrt{q}]$ by the Hasse theorem. The spectrum of the number of $\mathbb{F}_q$-rational points of an elliptic curve defined over $\mathbb{F}_q$ was determined by Waterhause;  see \cite[Section 9.9]{HKT}. For $q=2^h$, there exists an elliptic cubic over $\mathbb{F}_q$
with precisely $N_q=q+1 -m$ rational points$,$ where $|\,m\,|\leq
2\sqrt{q},$ for only $m\equiv 1 \pmod{2}$, $m=0$, $m=\pm \sqrt{q}$ with $h$ even, and $m=\pm \sqrt{2q}$ with $h$ odd. Let $AG(2,\mathbb{F}_q)$ be the affine plane whose projective closure is $PG(2,\mathbb{F}_q)$. If the intersection of $C_3$ with the line $\ell_\infty$ at infinity contains  at least two points defined over $\mathbb{F}_q$ and $|C_3\cap \ell_\infty|=2,3$ then the number of $\mathbb{F}_q$-rational point of $C_3$ in $AG(2,\mathbb{F}_q)$ belongs to the interval $[q-2\sqrt{q}-2,q+2\sqrt{q}-1]$.    

\subsection{Finite fields of order $q=2^h$}
Some basic facts concerning the trace function in finite fields of even characteristic are used in the proofs.
In the finite field $\mathbb{F}_q$ of order $q=2^h$, the \emph{trace} of an element is defined to be
$\mathfrak{Tr}(x)=x+x^2+x^4+\ldots+x^{2^{h-1}}$. Let $\mathcal{C}_0=\{x\in\mathbb{F}_q|\mathfrak{Tr}(x)=0\}$ and $\mathcal{C}_1=\{y\in\mathbb{F}_q|\mathfrak{Tr}(y)=1\}.$ Then $\mathcal{C}_0\cup \mathcal{C}_1=\mathbb{F}_2$ and  an equation $X^2+X+v=0$ with $v\in \mathbb{F}_q$ has either two, or zero solutions in $\mathbb{F}_q$ according as $v\in \mathcal{C}_0$, or $v\in \mathcal{C}_1$. For any $x\in \mathbb{F}_q$, there exists  $t\in\mathbb{F}_q$ such that $x=t^2+t$ if and only if $x\in\mathcal{C}_0$. For details the reader is referred to \cite{HJWP}.

\subsection{Linear codes}
A \emph{linear code} $ C $ of length $ n $ over a finite field $\mathbb{F}_q$ is a subspace of the vector space $\mathbb{F}_q^n$ over $\mathbb{F}_q$. The vectors in $C$ are the \emph{codewords}, and  if $C$ has dimension $k$ then it is a linear code of \emph{dimension} $k$. Fix a basis of $\mathbb{F}_q^n$. The \emph{weight} of a codeword is the number of its non-zero coordinates (entries). The \textit{Hamming distance} of two codewords $u,v\in C$ is the weight of $u-v $. The \textit{minimum distance} $d$ of a code $ C $ is the minimum of distances of all two distinct codewords of $ C$ or, equivalently, the minimum weight of the non-zero vectors of $C$. A $[n,k,d]_q $\emph{-code} is a linear code with above parameters $n,k,d$.

One may ask whether a given code is a ``good" one compared to others. Such a comparison of linear codes sharing  the same length and dimension is usually done with respect to their minimum distances, and in some cases by two further parameters, namely the \emph{relative distance} $\delta=d/n$ and the \emph{information} or \emph{dimension rate}  $R=k/n$. Codes with higher rates are considered to be better than codes with lower rates.

\subsubsection{The Datta-Johnsen code and its reduction}
\label{dj}The elementary symmetric polynomials $\sigma_m^i$ together with their $\mathbb{F}_q$-linear combinations form an ($m+1$)-dimensional $\mathbb{F}_q$-subspace in $\mathbb{F}_q[X_1,\ldots,X_m]$. Evaluating these polynomials on the set of all distinguished points in $\mathbb{F}_q^m$ (i.e. on all points with pairwise distinct coordinates in $\mathbb{F}_q^m$) is the \emph{Datta-Johnsen code} $C_m$ introduced in \cite{datta2023codes}. For $m < q$, the code $C_m$ is a non-degenerate $[n,k,d]_q$ code, where $n =P(q,m)$ with
\begin{equation*}
P(q,m)=
\begin{cases}
\binom{q}{m} m! \quad \text{if} \ \ m \le q,\\
\,\,\,\,0 \quad \text{otherwise},
\end{cases}
\end{equation*} $k=m+1$ and $d =(q-m)P(q-1,m-1)$; see \cite[Proposition 3.2]{datta2023codes}. In \cite[Remark 3.3]{datta2023codes}, the authors pointed out that the distinguished points are partitioned into $\binom{q}{m}$ subsets each of which is an orbit of the symmetric group of degree $m$. Therefore, a smaller evaluation code $C_m'$ can be obtained by evaluating symmetric polynomials on an ordered set $\mathcal{Q}$ of representatives of those orbits. Such codes are named \emph{reduced Datta-Johnsen codes}. For $m < q$ , $C_m'$ is a non-degenerate $[N, K, D]_q$ linear code where $N = \binom{q}{m}$, $K=m+1$ and $D=\binom{q}{m}- \binom{q-1}{m-1}$; see \cite[Proposition 3.4]{datta2023codes}.
\section{The pointset $\bar{\Delta}$ covered by all parabolas of equation $y=ax^2$ with $\rm{Tr}(a)=0$}\label{dede}

In the affine plane $\Pi\cong AG(2,\mathbb{F}_q)$ with coordinates $(Y_1,Y_2)$, the parabolas $\mathcal{P}_a$ of equation $Y_2=aY_1^2$ with $\mathfrak{Tr}(a)=0$ cover the pointset
\begin{equation}
    \label{delta}
   \bar{\Delta}=\{(y_1,a y_1^2)|y_1\in\mathbb{F}_q, \mathfrak{Tr}(a)=0\}.
\end{equation}
Note that for $a=0$, $\mathcal{P}_a$ is the line $Y_2=0$. Since $\bar{\Delta}=\Delta\cup \{O\}$ with $\Delta$ introduced in Section \ref{sec1}, we have
\begin{equation*}\label{delta2}
     \bar{\Delta}=\{(x,(t^2+t) x^2)|x,t\in\mathbb{F}_q\}.
     \end{equation*}
As exactly one half of the elements in $\mathbb{F}_q$ have zero trace, $\bar{\Delta}$ has size $\frac{1}{2} q(q-1)+1$.

Take another affine plane $\Sigma\cong AG(2,\mathbb{F}_q)$ with coordinates $(X_1,X_2)$. Let $\varphi_2$ be the map from $\Sigma$ to $\Pi$ which takes the point $(x_1,x_2)$ to the point $(y_1,y_2)$ with $y_1=x_1+x_2$ and $y_2=x_1x_2$. This map is not surjective as the points $(x_1,x_2)$ and $(x_2,x_1)$ have the same image. More precisely, we show that $\varphi_2$ takes $\Sigma$ into the set covered by $\bar{\Delta}$ together with the line $Y_1=0$. If $(y_1,y_2)$ is the image of $(x_1,x_2)$ by $\varphi_2$, then
$x_1,x_2$ are the roots of the quadratic polynomial $X^2+y_1X+y_2$. If $y_1\ne 0$, replace $X$ by $Z=X/y_1$. Then the arising polynomial
$$Z^2+Z+\frac{y_2}{y_1^2}$$
has two roots in $\mathbb{F}_q$. Therefore $$\mathfrak{Tr}\left(\frac{y_2}{y_1^2}\right)=0,$$
see \cite[Section 1.4]{HJWP}.
Setting $a=y_2/y_1^2$, the image of $(x_1,x_2)$ is on the parabola of equation $Y_2=aY_1^2$ with $\mathfrak{Tr}(a)=0$.
The converse also holds. Since the image of the line of equation $X_1=X_2$ is the line of equation $Y_1=0$, the claim is proven. It may be noticed that the images of the distinguished points of $\Sigma$ are exactly the points of $\bar{\Delta}$ outside the line $Y_1=0$.

As we have  pointed out in Section \ref{sec1}, $\Delta$ may be viewed as the analog of the set of all external points to a parabola in a plane of odd order.
\subsection{Intersection of a line with $\bar{\Delta}$}
\label{linesec}
First the case where $\ell$ is a non-vertical line is considered. In this case, $\ell$ has equation $Y_2=mY_1+b$. If $m=0,b=0$, then $\ell$ is contained in $\bar{\Delta}$, and hence $|\ell\cap \bar{\Delta}|=q$. Otherwise, in the map $\varphi_2$, $\ell$ is the image of the hyperbole $\mathcal{H}$ with equation $X_1X_2-m(X_1+X_2)+b=0$. Since the line $X_1=X_2$ meets $\mathcal{H}$ in exactly one point, each of the remaining $q-2$ points of $\mathcal{H}$ is a distinguished point. Therefore, if $b\ne0$, $|\ell\cap \bar{\Delta}|=\frac{1}{2} (q-2)$. If $b=0$, the line $X_1=X_2$ meets $\mathcal{H}$ in $O$ and $|\ell\cap\bar{\Delta}|=\frac{1}{2}(q-2)+1=\frac{1}{2}q-2$. Now, the case where $\ell$ is a vertical line is considered. Then $\ell$ has equation $Y_1=b$ and $\ell$ is the image by the map $\varphi_2$ of the line $\ell_b: X_1+X_2=b$. If $b\neq0$, then $\ell_b$ is parallel to the line $X_1=X_2$. Therefore $|\ell\cap\bar{\Delta}|=\frac{1}{2}q$. If $b=0$, then $\ell_b$ coincides with the line $X_1=X_2$ and $\ell\cap\bar{\Delta}=\{O\}$.

\subsection{Intersection of a conic with $\bar{\Delta}$}
\label{conicsec}
Let $C$ be a conic of Equation (\ref{eq21072025}). From now on we suppose that the triples
\begin{equation}
\label{eq21112025} (a_{11},a_{12},a_{22}), (a_{11},a_{13},a_{33}), (a_{11},a_{22},a_{23}), (a_{12},a_{13},a_{23}),(a_{22},a_{23},a_{33}) 
\end{equation}
are non-trivial. 
Let $\mathcal F$ be the (possible reducible) curve of equation
\begin{equation*}
\label{eq21072025A}
F(X,T)=a_{11}X^2+a_{12}(T^2+T)X^3+a_{22}(T^4+T^2)X^4+a_{13}X+a_{23}(T^2+T)X^2+a_{33}=0
\end{equation*}
defined in the affine plane $AG(2,\mathbb{F}_q)$ with coordinates $(X,T)$, and also viewed as a curve in $AG(2,\mathbb{K})$. If $a_{33}\ne 0$ then the line $X=0$ is not a component of $\mathcal{F}$, and in this case we put $\mathcal{F}^{(0)}=\mathcal{F}$. If $a_{33}=0$ then $\mathcal{F}$  splits into two components, namely a (possibly reducible) curve $\mathcal{F}^{(s)}$ and the line of equation $x=0$ counted with multiplicity $s$ where $s=1$ for $a_{13}\ne 0$, and $s=2$, for $a_{13}=0,a_{11}\ne 0$. The equations of $\mathcal{F}^{(s)}$ are
\begin{equation*}
\label{eq21072025A0}
F^{(0)}(X,T)=a_{11}X^2+a_{12}(T^2+T)X^3+a_{22}(T^4+T^2)X^4+a_{13}X+a_{23}(T^2+T)X^2+a_{33}=0,\,\, a_{33}\ne 0,
\end{equation*}
\begin{equation*}
\label{eq21072025AA}
F^{(1)}(X,T)=a_{11}X+a_{12}(T^2+T)X^2+a_{22}(T^4+T^2)X^3+a_{13}+a_{23}(T^2+T)x=0,\,\, a_{13}\ne 0,
\end{equation*}
and
\begin{equation*}
\label{eq21072025AAA}
F^{(2)}(X,T)=a_{11}+a_{12}(T^2+T)X+a_{22}(T^4+T^2)X^2+a_{23}(T^2+T)=0,\,\,  a_{11}\ne 0,
\end{equation*}
respectively.
Notice that $\mathcal{F}$ and $\mathcal{F}^{(s)}$ share the same points outside the line $X=0$.

The following lemma shows that the number of common points of $\bar{\Delta}$ and $C$ is determined by the number of points of $\mathcal{F}$ in $AG(2,\mathbb{F}_q)$. 
\begin{lemma}
\label{lem27102025}
$$|\bar{\Delta}\cap C|=
\begin{cases}
{\mbox{$\textstyle{\frac{1}{2}}N_q(\mathcal{F})$, for $a_{33}\ne 0$;}}\\
{\mbox{$1+\textstyle{\frac{1}{2}}N_q(\mathcal{F}^{(1)})$, for $a_{33}=0$ and $a_{13}\ne 0$;}}\\
{\mbox{$1+\textstyle{\frac{1}{2}}N_q(\mathcal{F}^{(2)})$, for $a_{33}=a_{13}=0$, and either $a_{23}=0$ or $\mathfrak{Tr}(a_{11}/a_{23})=1;$}}\\
{\mbox{$\textstyle{\frac{1}{2}}N_q(\mathcal{F}^{(2)})$, for $a_{33}=a_{13}=0$, $a_{23}\ne 0$, and $\mathfrak{Tr}(a_{11}/a_{23})=0.$}}\\
\end{cases}
$$
\end{lemma}
\begin{proof} The quadratic transformation $\psi$ given by $(X,T)\to (X,Y)$  with $Y=(T^2+T)$ $X^2$
takes the point $(x,t)\in \mathcal F$ to the point
$(x,(t^2+t)x^2)\in \bar{\Delta}\cap C$. More precisely, $\psi$, viewed as a map from the set of all points of $\mathcal{F}$ into $\bar{\Delta}\cap C$, is an almost surjective map. In fact,  every point $(x,(t^2+t)x^2)\in \bar{\Delta}\cap C$ with $x\ne 0$ is the image of exactly two points which are $(x,t)$ and $(x,t+1)$ of $\mathcal{F}$ while, for $x=0$, $\bar{\Delta}\cap C$ has a unique point, namely $(0,0)$, which may be the image of some points $(0,t)\in \mathcal{F}$, but this only occurs when $a_{33}=0$, and the line $X=0$ is a component of $\mathcal{F}$. For $a_{33}\ne 0$, we have $\mathcal{F}=\mathcal{F}^{(0)}$, and the claim follows from the fact that $\mathcal{F}^{(0)}$ has no point on the line $X=0$. If $a_{33}=0$, then $C$ passes through $O=(0,0)$ but $\mathcal{F}^{(1)}$ does not. Therefore, replacing $\mathcal F$ by $\mathcal{F}^{(1)}$ in the above argument shows the claim concerning $\mathcal{F}^{(1)}$. If both $a_{33}$ and $a_{13}$ vanish, we look at $\mathcal{F}^{(2)}$. The claim for this case follows from the fact that  $\mathcal{F}^{(2)}$ contains no point on the line $X=0$ when either $a_{23}=0$, or $\mathfrak{Tr}(a_{11}/a_{23})=1$, and it contains exactly two points when $a_{23} \ne 0$ and  $\mathfrak{Tr}(a_{11}/a_{23})=0$.
\end{proof}

Therefore, we need to count the points of $\mathcal{F}$. Our counting will be done on another curve birationally equivalent to $\mathcal{F}$.

The quadratic transformation $\omega$ given by $\omega:(X,T)\mapsto (X,V)$ with $V=XT$ is injective in the affine plane $AG(2,\mathbb{F}_q)$ with coordinates $(X,T)$ except for the points $(x,t)$ with $x=0$.
Also, $\omega$ is not surjective, as
${\rm{Im}}(\omega)$ in the affine plane $AG(2,\mathbb{F}_q)$ with coordinates $(X,V)$ consists of all points other than those on the $V$-axis and plus the origin.

Moreover, the image $\mathcal{G}$ of the curve $\mathcal{F}$ is the curve of equation
\begin{equation*}
\label{eq21072025BO}
G(X,V)=a_{11}X^2+a_{12}(XV^2+X^2V)+a_{22}(V^4+V^2X^2)+a_{13}X+a_{23}(V^2+VX)+a_{33}=0.
\end{equation*}
In particular, $\omega$ transforms $\mathcal{F}^{(0)}$ into the  curve
$\mathcal{G}^{(0)}$ of equation
\begin{equation*}
\label{eq21072025B}
G^{(0)}(X,V)=a_{11}X^2+a_{12}(XV^2+X^2V)+a_{22}(V^4+V^2X^2)+a_{13}X+a_{23}(V^2+VX)+a_{33}=0,\,\, a_{33}\ne 0,
\end{equation*}
Also, for $s=1,2$, $\omega$ transforms $\mathcal{F}^{(s)}$ to the curve $\mathcal{G}^{(s)}$ of equations
\begin{equation*}
\label{eq21072025BA}
G^{(1)}(X,V)=a_{11}X^2+a_{12}(XV^2+X^2V)+a_{22}(V^4+V^2X^2)+a_{13}X+a_{23}(V^2+VX)=0,\,\, a_{13}\ne 0,
\end{equation*}
\begin{equation*}
\label{eq21072025BB}
G^{(2)}(X,V)=a_{11}X^2+a_{12}(XV^2+X^2V)+a_{22}(V^4+V^2X^2)+a_{23}(V^2+VX)=0,\,\, a_{11}\ne 0,
\end{equation*}
respectively.
For $s=0,1,2$, the curve $\mathcal{G}^{(s)}$ may contain some points on the line of equation $X=0$, as $\mathcal{G}^{(s)}$ has zero, or one, or two points according as the polynomial $a_{22}V^4+a_{23}V^2+a_{33}$ has zero, or one, or two roots. More precisely, this number equals one if and only if exactly one of $a_{22}$ and $a_{23}$  vanishes, while it is zero when  either $a_{22}=a_{23}=0$ (and $a_{33}\ne 0$ by hypothesis), or  
$\mathfrak{Tr}(b)=1$, and it is two  when $\mathfrak{Tr}(b)=0$ where  $b=a_{33}a_{22}/a_{23}^2.$
Since $\mathcal{F}^{(s)}$ has no point on the line $X=0$ for $s=0,1$, its number of points may change a bit  under the action of $\omega$.
\begin{equation}\label{eq19112025}
    N_q(\mathcal{F}^{(0)})=
\begin{cases}
 {\mbox{$N_q(\mathcal{G}^{(0)})$, for $a_{23}\ne 0$, $a_{22}\ne 0$ and $\mathfrak{Tr}(a_{22}a_{33}/a_{23}^2)=1$}};\\
  {\mbox{$N_q(\mathcal{G}^{(0)})$, for $a_{22}=a_{23}=0$}};\\
 {\mbox{$N_q(\mathcal{G}^{(0)})-1$, when exactly one of $a_{22}=0$ and $a_{23}=0$ holds;}}\\
 {\mbox{$N_q(\mathcal{G}^{(0)})-2$, for $a_{23}\ne 0$, $a_{22}\ne 0$ and $\mathfrak{Tr}(a_{22}a_{33}/a_{23}^2)=0$.}}\\
\end{cases}
\end{equation}
\begin{equation}\label{eq20112025}
    N_q(\mathcal{F}^{(1)})=
\begin{cases}
 {\mbox{$N_q(\mathcal{G}^{(1)})-1$, for $a_{23}=0$, or $a_{22}=0$;}}\\
 {\mbox{$N_q(\mathcal{G}^{(1)})-2$, for $a_{23}\ne 0$ and $a_{22}\ne 0$
 .}}\\
\end{cases}
\end{equation}
\begin{equation}\label{eq20112025A}
    N_q(\mathcal{F}^{(2)})=
\begin{cases}
 {\mbox{$N_q(\mathcal{G}^{(2)})-1$, for either $a_{23}=0$ or $\mathfrak{Tr}(a_{11}/a_{23})=1$ and $a_{22}=0$;}}\\
 {\mbox{$N_q(\mathcal{G}^{(2)})+1$, for $a_{23}\ne 0$, $\mathfrak{Tr}(a_{11}/a_{23})=0$ and $a_{22}=0$;}}\\
 {\mbox{$N_q(\mathcal{G}^{(2)})-2$, for $a_{23}\ne 0$, $\mathfrak{Tr}(a_{11}/a_{23})=1$ and $a_{22}\ne 0$;}}\\
 {\mbox{$N_q(\mathcal{G}^{(2)})$, for $a_{23}\ne 0$, $\mathfrak{Tr}(a_{11}/a_{23})=0$ and $a_{22}\ne 0$.}}\\
\end{cases}
\end{equation}
\begin{proposition}
\label{pro02012026} Assume that $a_{12}=a_{22}=a_{23}=0$. Then $a_{11}a_{13}a_{33}\neq 0$, and 
$$|\Delta\cap C|=|\bar{\Delta}\cap C|=N_q(\mathcal{F})=N_q(\mathcal{G})=
\begin{cases}
{\mbox{$q$, for $\mathfrak{Tr}(a_{11}a_{33}/a_{13}^2)=0$}};\\
{\mbox{$0$, for $\mathfrak{Tr}(a_{11}a_{33}/a_{13}^2)=1$}}.
\end{cases}
$$
%$|\Delta|=q$.
\end{proposition}
\begin{proof} In this case, $\mathcal{F}(X,T)=\mathcal{F}^{(0)}(X,T)=a_{11}X^2+a_{13}X+a_{33}$.  From hypothesis (\ref{eq21112025}),  none of $a_{11},a_{13}$ and $a_{33}$ vanishes. Therefore, 
either $N_q(\mathcal{F})=q$, or$N_q(\mathcal{F}_q)=0$ according as $\mathfrak{Tr}(b)=0$, or $\mathfrak{Tr}(b)=1$ where $b=a_{11}a_{33}/a_{13}^2$.  
\end{proof}
Furthermore, a straightforward argument relying on the above equations  proves the following claims.
\begin{proposition}
\label{pro12082025B1}
Assume that $a_{12}=a_{22}=0,a_{23}\ne 0$, and let $a_{23}=1$. Then $\mathcal{G}^{(s)}$ is a conic which degenerates if and only if $a_{33}=a_{13}^2$. If  $\mathcal{G}^{(s)}$ is non-degenerate then it is either a hyperbole or an ellipse not passing through $O$, according as $\mathfrak{Tr}(a_{11})=0$, or $\mathfrak{Tr}(a_{11})=1$, and accordingly
$$N_q(\mathcal{G}^{(s)})=
\begin{cases}
{\mbox{$q-1$, for $\mathfrak{Tr}(a_{11})=0$}};\\
{\mbox{$q+1$, for $\mathfrak{Tr}(a_{11})=1$}}.
\end{cases}
$$
If $\mathcal{G}^{(s)}$ is degenerate then
$$G^{(s)}(X,V)=((\alpha+1)X+V+a_{13})(\alpha X+V+a_{13})$$ with $\alpha^2+\alpha+a_{11}=0$, and $\mathcal{G}^{(s)}$ splits into two distinct  lines through the point $(0,a_{13}),$ which are defined over $\mathbb{F}_q$ if and only if $\mathfrak{Tr}(a_{11})=0$, and in this case
$N_q(\mathcal{G}^{(s)})=2q-1$, otherwise  $\mathfrak{Tr}(a_{11})=1$ and $N_q(\mathcal{G}^{(s)})=1$.
\end{proposition}
It should be noticed that when $\mathcal{G}$ happens to be a degenerate conic, this does not imply that $C$ is also degenerate.

Lemma \ref{lem27102025} together with (\ref{eq19112025}), (\ref{eq20112025}), (\ref{eq20112025A}) and Proposition \ref{pro12082025B1} give the following result.

\begin{theorem}
    \label{pro19112025AA} Let $C$ be a non-degenerate conic of equation
$$a_{11}X^2+a_{12}XY+a_{22}Y^2+a_{13}X+a_{23}Y+a_{33}=0.$$ Assume that $a_{12}=a_{22}=0,a_{23}\ne 0$, and  let $a_{23}=1$. Then
$$ |\bar{\Delta}\cap C|=
\begin{cases}
\textstyle{\frac{1}{2}}q-1,\,\,{\mbox{for $\mathfrak{Tr}(a_{11})=0$, $a_{33}\ne 0$  and $a_{33}\ne a_{13}^2$; }} \\
\textstyle{\frac{1}{2}}q,\,\,\quad\,\,\,\, {\mbox{for $\mathfrak{Tr}(a_{11})=1$, $a_{33}\ne 0$  and $a_{33}\ne a_{13}^2$; }} \\
\textstyle{\frac{1}{2}}q,\,\,\quad\,\,\,\,{\mbox{for $\mathfrak{Tr}(a_{11})=a_{33}=0$ and $a_{33}\ne a_{13}^2$; }} \\
\textstyle{\frac{1}{2}}q+1,\,\,{\mbox{for $\mathfrak{Tr}(a_{11})=1$, $a_{33}= 0$  and $a_{33}\ne a_{13}^2$; }} \\
q,\,\,\quad\,\,\,\,\,\,\,\,{\mbox{for $\mathfrak{Tr}(a_{11})=0$,   and $a_{33}=a_{13}=0$;}} \\
1,\,\,\quad\,\,\,\,\,\,\,\,{\mbox{for $\mathfrak{Tr}(a_{11})=1$,   and $a_{33}=a_{13}=0$;}} \\
q-1,\,\,\,\,\,\,{\mbox{for $\mathfrak{Tr}(a_{11})=0$,   and $a_{33}=a_{13}^2, a_{33}\ne 0$;}} \\
0,\,\,\quad\,\,\,\,\,\,\,\,{\mbox{for $\mathfrak{Tr}(a_{11})=1$,   and $a_{33}=a_{13}^2, a_{33}\ne 0$.}}
\end{cases}
$$
\end{theorem}

To study $\mathcal{G}$ in the cases uncovered by Theorem \ref{pro19112025AA}, we need some properties over the projective plane $PG(2,\mathbb{F}_q)$ which is the projective closure of $AG(2,\mathbb{F}_q)$. For this purpose, $PG(2,\mathbb{F}_q)$ will be equipped by homogeneous coordinates $(X_1:X_2:X_3)$ where
$V=X_1/X_3,X=X_2/X_3$.

Assume that $C$ is irreducible and that 
\begin{equation}
\label{eq17112025}
{\mbox{\emph{either} $a_{12}\ne 0$, \emph{or} $a_{22}\ne 0$}}.
\end{equation}
It is straightforward to verify that $\mathcal{G}$ has two singular points, namely $X_\infty=(0:1:0)$ and $Q_\infty=(1:1:0)$, both doubly points with tangents different from the line at infinity.
Take $\bar{v}$ from $\mathbb{F}_q$ or from its algebraic closure $\mathbb{F}$ such that $a_{11}+a_{12}\bar{v}+a_{22}\bar{v}^2=0$.
Then the line $\ell$ of equation $V=\bar{v}$ is a tangent  to $\mathcal{G}$ at $X_\infty$.
Consider the linear system $\Sigma_2$ generated by all conics through $X_\infty$ and $Q_\infty$ which are tangent to $\ell$. Then $\Sigma_2$
consists of all conics $C_{(\alpha:\beta:\gamma)}$ of equations $\alpha(XV+V^2+\bar{v}X)+\beta (V+\bar{v})+\gamma=0$
where $(\alpha:\beta:\gamma)$ runs over all points of $PG(2,\mathbb{F}$). By construction, $\Sigma_2$ cuts out on $\mathcal{G}$ a simple, fixed point free $2$-dimensional linear series of degree $3$.
Therefore, the quadratic transformation $\omega_1$ given by $(X,V)\mapsto (XV+V^2+\bar{v}X,V+\bar{v})$ transforms $\mathcal{G}$ into a curve $\mathcal{H}$ of degree $3$.
Furthermore, $\omega_1$ is neither injective, as the points on the line $\ell_{\bar{v}}$ of equation $V=\bar{v}$ are mapped into the same point $(\bar{v}^2,0)$, nor surjective as the points on the line of equation $V=0$ are not in $Im(\omega_1$) apart from $(\bar{v}^2,0)$.
Moreover, $\omega_1$ is invertible on the points outside $\ell_{\bar{v}}$, as $\theta_1\circ \omega_1$ is the identity transformation outside $\ell_{\bar{v}}$ where $\theta_1$ is defined by $(X,V)\mapsto ((X+V^2+\bar{v}^2)V^{-1},V+\bar{v})$.

A direct computation shows that $\mathcal{H}$ has equation $H(X,V)=0$ with
\begin{equation}
\label{eq09082025}
\begin{array}{llll}H(X,V)=&&(Va_{22}+a_{12})X^2+(V^2a_{12}+Va_{23}+\bar{v}^2a_{12}+\bar{v}a_{23}+a_{13})X+\\
&&(\bar{v}a_{23}+a_{13})V^2+Va_{33}+\bar{v}^3a_{23}+\bar{v}^2a_{13}.
\end{array}
\end{equation}
Equivalently,
\begin{equation}
\label{eq09082025A}
\begin{array}{llll}H(X,V)=&&(a_{12}X+\bar{v}a_{23}+a_{13})V^2+(X^2a_{22}+Xa_{23}+a_{33})V+\\
&& (X+\bar{v}^2)(a_{12}X+\bar{v}a_{23}+a_{13}).
\end{array}
\end{equation}
Moreover, $\mathcal{H}$ has three points at infinity: $V_\infty=(1:0:0)$, $X_\infty=(0:1:0)$ and $Q_\infty=(a_{22}:a_{12}:0)$. Here,
$X_\infty=Q_\infty$ if and only if $a_{22}=0$, and $V_\infty=Q_\infty$ if and only if $a_{12}=0$.
We show that $X_\infty$ is not a singular point of $\mathcal{H}$. 
Assume on the contrary that $X_\infty$ is a double point of $\mathcal{H}$. Then $X_\infty=Q_\infty$. Since $V_\infty\in\mathcal{H}$, the line $\ell_\infty$ is not a tangent to $\mathcal{H}$ at $X_\infty$. Therefore, there exists $v^*\in \mathbb{F}_q$ such that the line  $V=v^*$ is tangent to $\mathcal{H}$ at $X_\infty$. Since $\deg(\mathcal{H})=3$ and $X_\infty$ is a double point, no affine point of the line $V=v^*$ lies on $\mathcal{H}$. Therefore, the coefficient of $X^2$ in (\ref{eq09082025}) vanishes for $V=v^*$. Since $a_{22}=0$, this yields $a_{12}=0$ and (\ref{eq17112025}) does not hold. 
Therefore $X_\infty$ is a non-singular point of $\mathcal{H}$. A similar argument can be used to show that neither $V_\infty$ is a singular point. In fact, if $V_\infty$ is a double point, then the coefficient of  $V^2$ in (\ref{eq09082025A}) vanishes. Therefore, 
$\bar{v}a_{23}+a_{13}=0$. Since $a_{12}=0$, we also have $a_{11}+a_{22}\bar{v}^2=0$. Since the triple $(a_{12},a_{13},a_{23})$ is non-trivial, this yields that (\ref{conde}) holds. Therefore, if $C$ is irreducible, then $V_\infty$ is a non-singular point of $\mathcal{H}$. Clearly, if $Q_\infty\ne X_\infty$ and $Q_\infty\ne V_\infty$, then $Q_\infty$ is a not a singular point of $\mathcal{H}$, as well.  

Since $\mathcal{G}$ has at most one point $(x,\bar{v})$, its number of points does not change under the action of $\omega_1$, that is, $\mathcal{G}$ and $\mathcal{H}$ have the same number of points.

We are in a position to determine when the curve $\mathcal{H}$ and in turns the curve $\mathcal{G}$ is reducible.

First the general case is investigated.
\begin{proposition}
\label{pro12082025} Assume that $a_{12}\ne 0$ and $a_{22}\ne 0$. Then $C$ is degenerate if and only if $\mathcal{H}$ is reducible. Moreover, if $\mathcal{H}$ is reducible then it splits into three linear components.
\end{proposition}
\begin{proof}
Let $\ell$ be a line and suppose $\ell$ to be a linear component of $\mathcal{H}$. Then $\ell$ passes through one of the points $V_\infty$, $X_\infty$ and $Q_\infty$.

(i)
If $V_\infty\in \ell$, let $\ell_V=\ell$. Then $\ell_V$ has equation $X-x=0$ for some $x\in \mathbb{F}$. Therefore, $\ell_V$ is a component of $\mathcal{H}$ if and only if $x$ is a solution of the system in the indeterminate $Z$ arising from (\ref{eq09082025A})
\begin{equation}
\label{eq09082025AA}
\begin{cases}
a_{12}Z+\bar{v}a_{23}+a_{13}=0;\\
a_{22}Z^2+a_{23}Z+a_{33}=0;\\
(Z+\bar{v}^2)(a_{12}Z+\bar{v}a_{23}+a_{13})=0.
\end{cases}
\end{equation}
 The first equation in (\ref{eq09082025AA}) implies the third one.
 Moreover, eliminating $Z$ from the first two equations gives $U_{12}=0$ where
\begin{equation}
\label{eq12082025}
U_{12}=\bar{v}^2a_{22}a_{23}^2+\bar{v}a_{12}a_{23}^2+a_{12}^2a_{33}+a_{12}a_{13}a_{23}+a_{22}a_{13}^2.
\end{equation}
Since $a_{11}+a_{12}\bar{v}+a_{22}\bar{v}^2=0$, it turns out from (\ref{eq12082025}) that the system of the first two equations has a solution if and only if
$$a_{11}a_{23}^2+a_{12}a_{23}a_{13}+a_{22}a_{13}^2+a_{33}a_{12}^2=0,$$
that is, $V_\infty\in \ell_V$ if and only if the conic $C$ is reducible.

(ii) If $Q_\infty \in \ell$, let $\ell_Q=\ell$. Then $\ell_Q$ is parallel to the line  of equation $a_{22}X+a_{12}V=0$. Replacing $a_{22}X+a_{12}V$ by $W$, Equation (\ref{eq09082025A}) of $\mathcal{H}$ becomes
\begin{equation}
\label{eq09082025w}
\begin{array}{llll}L(X,W)=&&(Xa_{12}+\bar{v}a_{23}+a_{13})W^2+a_{12}(X^2a_{22}+Xa_{23}+a_{33})W+\\
&&X^2(a_{12}^3+\bar{v}a_{22}^2a_{23}+a_{11}a_{22}a_{33}+a_{22}^2a_{13})+\\
&&a_{12}x(\bar{v}^2a_{12}^2+\bar{v}a_{12}a_{23}+a_{12}a_{13}+a_{22}a_{33})+\\
&& \bar{v}^2a_{12}^2(\bar{v}a_{23}+a_{13}).
\end{array}
\end{equation}
Also, the equation of $\ell_Q$ becomes $X=w$ for some $w\in \mathbb{F}$.
Therefore, $\ell_Q$ is component of $\mathcal{H}$ if and only if $w$ is a solution of the system in the indeterminate $Z$ arising from (\ref{eq09082025w})
\begin{equation}
\label{eq09082025WW}
\begin{cases}
a_{12}Z+\bar{v}a_{23}+a_{13}=0;\\
a_{12}(a_{22}Z^2+a_{23}Z+a_{33})=0;\\
(a_{12}^3+\bar{v}a_{22}^2a_{23}+a_{11}a_{22}a_{33}+a_{22}^2a_{13})
Z^2+\\
a_{12}(\bar{v}^2a_{12}^2+\bar{v}a_{12}a_{23}+a_{12}a_{13})Z+a_{11}a_{22}a_{33}+
a_{12}^2(\bar{v}^3a_{23}+\bar{v}a_{13})=0.
\end{cases}
\end{equation}
The system of the first two equations has a solution if and only if
$Z_{12}=0$ where
\begin{equation}
\label{eq12082025B}
Z_{12}=\bar{v}^2a_{22}a_{23}^2+\bar{v}a_{12}a_{23}^2+a_{12}^2a_{33}+a_{12}a_{13}a_{23}+a_{22}a_{13}^2,
\end{equation}
and that of the first and the third ones if and only if $Z_{13}=0$ where
\begin{equation}
\label{eq12082025BB}
Z_{13}=(\bar{v}a_{23}+a_{13})Z_{12}.
\end{equation}
Therefore, $\ell$ is a component of $\mathcal{H}$ through $Q_\infty$ if and only if $Z_{12}=0$. Since $Z_{12}=a_{22}U_{12}$, it turns out that
$\ell_Q$ is a component of $\mathcal{H}$ if and only if so is $\ell_V$. Moreover, this is the case if and only if $C$ is reducible.

Before dealing with the last case $X_\infty\in \ell$, we rewrite the above arguments starting off with Equation (\ref{eq09082025}).
Replacing $a_{22}X+a_{12}V$ by $W$, Equation (\ref{eq09082025}) of $\mathcal{H}$ becomes
\begin{equation*}
\label{eq09082025w1}
\begin{array}{llll}L(V,W)=&&(Va_{22}+a_{12})W^2+a_{22}(V^2a_{12}+Va_{23}+\bar{v}^2a_{12}+\bar{v}a_{23}+a_{13})W+\\
&&V^2(\bar{v}a_{22}^2a_{23} + a_{12}^3 + a_{12}a_{22}a_{23} + a_{22}^2a_{13}) +  \\ && V(\bar{v}^2a_{12}^2a_{22} +\bar{v}a_{12}a_{22}a_{23} +a_{12}a_{22}a_{13} + a_{22}^2a_{33}) + \bar{v}^3a_{22}^2a_{23} + \bar{v}^2a_{22}^2a_{13}.
\end{array}
\end{equation*}
Also, the equation of $\ell_X$ becomes $V=v$ for some $v\in \mathbb{F}$. 
The formulas analogous to (\ref{eq09082025WW}), (\ref{eq12082025B}), and (\ref{eq12082025BB}) are
\begin{equation*}
\label{eq09082025WR}
\begin{cases}
Za_{22}+a_{12}=0;\\
a_{22}(Z^2a_{12}+Za_{23}+\bar{v}^2a_{12}+\bar{v}a_{23}+a_{13})=0;\\
Z^2\bar{v}a_{22}^2a_{23} + Z^2a_{12}^3 + Z^2a_{12}a_{22}a_{23} + Z^2a_{22}^2a_{13} + Z\bar{v}^2a_{12}^2a_{22} + \\
Z\bar{v}a_{12}a_{22}a_{23} +Za_{12}a_{22}a_{13} + Za_{22}^2a_{33} + \bar{v}^3a_{22}^2a_{23} + \bar{v}^2a_{22}^2a_{13}=0;
\end{cases}
\end{equation*}
$Q_{12}=0$ where
\begin{equation*}
\label{eq12082025BQ}
Q_{12}=\bar{v}^2a_{12}a_{22}^3 + \bar{v}a_{22}^3a_{23}+ a_{12}^3a_{22} + a_{12}a_{22}^2a_{23} + a_{22}^3a_{13};
\end{equation*}
$Q_{13}=0$ where
\begin{equation*}
\label{eq12082025BBQ}
Q_{13}=\bar{v}^3a_{22}^4a_{23}+ \bar{v}^2a_{12}^3a_{22}^2 + \bar{v}^2a_{22}^4a_{13}+ a_{12}^5 + a_{12}^3a_{22}a_{23} + a_{12}a_{22}^3a_{33}.
\end{equation*}
Therefore, $\ell_Q$ is a component of $\mathcal{H}$ if and only if $Q_{12}=0$ and $Q_{13}=0$. If this is the case then $C$ is reducible.

(iii) If $X_\infty\in \ell$, let $\ell_X=\ell$. Then $\ell_X$ has equation $V=v$ for some $v\in \mathbb{F}$. Therefore, $\ell_X$ is a component of $\mathcal{H}$ if and only if $v$ is a solution of the system in the indeterminate $Z$ arising from (\ref{eq09082025})
\begin{equation*}
\label{eq09082025C}
\begin{cases}Za_{22}+a_{12}=0; \\
Z^2a_{12}+Za_{23}+\bar{v}^2a_{12}+\bar{v}a_{23}+a_{13}=0;\\
Z^2(\bar{v}a_{23}+a_{13})+Za_{33}+\bar{v}^3a_{23}+\bar{v}^2a_{13}=0.
\end{cases}
\end{equation*}
The system of the first two equations has a solution if and only if
$R_{12}=0$ where
\begin{equation*}
\label{eq12082025D}
R_{12}=\bar{v}^2 a_{12}a_{22}^2+\bar{v}a_{22}^2a_{23}+a_{12}^3+a_{12}a_{22}a_{23}+a_{22}^2a_{13},
\end{equation*}
and that of the first and the third ones if and only if $R_{13}=0$ where
\begin{equation*}
\label{eq12082025DD}
R_{13}=\bar{v}^3 a_{22}^2a_{23}+\bar{v}^2a_{22}^2a_{13}+\bar{v}a_{12}^2a_{23}+a_{12}a_{22}a_{33}+a_{12}^2a_{13}.
\end{equation*}
By a straightforward computation,
$${\mbox{$Q_{12}=a_{22}R_{12}$ and $Q_{13}=a_{22}^2R_{13}+a_{12}^2R_{12}$}}.$$

It turns out that $\ell_X$ is a component of $\mathcal{H}$ if and only if so is $\ell_Q$ (and hence $\ell_V$). Therefore, $C$ is reducible if and only if $\mathcal{H}$ is reducible, and if $\mathcal{H}$ is reducible then it splits into three lines.
\end{proof}
\begin{proposition}
\label{pro12082025A} Assume that $a_{12}=0$ but $a_{22}\ne 0$. Then $C$ is degenerate if and only if $\mathcal{H}$ is reducible. Moreover, if $\mathcal{H}$ is reducible then it splits into three linear components.
\end{proposition}
\begin{proof} Since $a_{12}=0$, $Q_\infty=V_\infty$.
Let $\ell$ be a line and suppose $\ell$ to be a linear component of $\mathcal{H}$. Then $\ell$ passes through one of the points $V_\infty$ and $X_\infty$.

(i)  If $V_\infty\in \ell$, let $\ell_V=\ell$, and  argue as in part (i) of the proof of Proposition \ref{pro12082025}.
Then $\ell_V$ is a component of $\mathcal{H}$ if and only if there exists a solution $\xi$ of the system in the indeterminate $Z$
\begin{equation*}
\label{eq09082025AAA}
\begin{cases}
\bar{v}a_{23}+a_{13}=0;\\
a_{22}Z^2+a_{23}Z+a_{33}=0;\\
(Z+\bar{v}^2)(\bar{v}a_{23}+a_{13})=0.
\end{cases}
\end{equation*}
Therefore $\ell_V$ is a component of $\mathcal{H}$ if and only if $\bar{v}a_{23}+a_{13}=0$. Actually, from (\ref{conde}) this may only occur when $C$ is degenerate as $a_{11}+\bar{v}^2a_{22}=0$. Moreover, since the triple $(a_{12},a_{13},a_{23})$ is assumed to be non-trivial, the polynomial $R(X)=a_{22}X^2+a_{23}X+a_{33}$ splits into two distinct linear factors. If $\xi_1$ and $\xi_2$ are the roots of $R(x)$, then the lines $X=\xi_1$ and $X=\xi_2$ are linear components of $\mathcal{H}$. Thus, $\mathcal{H}$ splits into three linear components where the third one is a line through $X_\infty$.

(ii) If $X_\infty\in \ell$, let $\ell_X=\ell$ and argue as in part (iii) of the proof of Proposition \ref{pro12082025}.  Then $\ell_X$ is a component of $\mathcal{H}$ if and only if there exists a solution $\nu$ of the system in the indeterminate $Z$
\begin{equation*}
\label{eq09082025CC}
\begin{cases}Za_{22}=0; \\
Za_{23}+\bar{v}a_{23}+a_{13}=0;\\
Z^2(\bar{v}a_{23}+a_{13})+Za_{33}+\bar{v}^3a_{23}+\bar{v}^2a_{13}=0.
\end{cases}
\end{equation*}
From the first equation, $\nu=0$. Then $\ell_X$ is a component of $\mathcal{H}$ if and only if $\bar{v}a_{23}+a_{13}=0$.
As in part (i), this may only occur when $C$ is degenerate. By (\ref{eq09082025A}),
the other two components of $\mathcal{H}$ come from the factorization of the above polynomial $R(X)$, and they are the lines $X=\xi_1$ and $X=\xi_2$.  
\end{proof}

\begin{proposition}
\label{pro12082025AA} Assume that $a_{22}=0$ but $a_{12}\ne 0$. Then $C$ is degenerate if and only if $\mathcal{H}$ is reducible. 
\end{proposition}
\begin{proof} Since $a_{22}=0$, we have $Q_\infty=X_\infty$.
Let $\ell$ be a line and suppose $\ell$ to be a linear component of $\mathcal{H}$. Then $\ell$ passes through one of the points $V_\infty$ and $X_\infty$.

(i)  If $X_\infty\in \ell$, let $\ell_X=\ell$, and  argue as in part (iii) of the proof of Proposition \ref{pro12082025}.
Therefore, $\ell_X$ is a component of $\mathcal{H}$ if and only if $\nu$ is a solution of the system in the indeterminate $Z$
\begin{equation*}
\label{eq09082025AAAA}
\begin{cases}a_{12}=0; \\
Z^2a_{12}+Za_{23}+\bar{v}^2a_{12}+\bar{v}a_{23}+a_{13}=0;\\
Z^2(\bar{v}a_{23}+a_{13})+Za_{33}+\bar{v}^3a_{23}+\bar{v}^2a_{13}=0.
\end{cases}
\end{equation*}
From the first equation we obtain a contradiction. Thus this case can not occur.

(ii) If $V_\infty\in \ell$, let $\ell_V=\ell$ and argue as in part (i) of the proof of Proposition \ref{pro12082025}.  Therefore, $\ell_V$ is a component of $\mathcal{H}$ if and only if there exists a solution $\xi$ of the system in the indeterminate $Z$
\begin{equation*}
\label{eq09082025CCC}
\begin{cases}
a_{12}Z+\bar{v}a_{23}+a_{13}=0;\\
a_{23}Z+a_{33}=0;\\
(Z+\bar{v}^2)(a_{12}Z+\bar{v}a_{23}+a_{13})=0.
\end{cases}
\end{equation*}
The system of the first two equations has a solution if and only if $S_{12}=0$ where
\begin{equation} \label{S12}
    S_{12}=a_{12}a_{33}+a_{23}^2\bar{v}+a_{13}a_{23}=0.
\end{equation}
Since $a_{11}+a_{12}\bar{v}=0$, it turns out form (\ref{S12}) that the system of the first two equations has a solution if and only if
$$a_{11}a_{23}^2+a_{12}a_{23}a_{13}+a_{33}a_{12}^2=0,$$
that is, $V_\infty\in\ell_V$ if and only if the conic $C$ is reducible.
\end{proof}
We are in a position to prove the following claim.
\begin{theorem}
\label{th05112025}
Let $C$ be a non-degenerate conic of equation
$$a_{11}X^2+a_{12}XY+a_{22}Y^2+a_{13}X+a_{23}Y+a_{33}=0.$$
Then 
\begin{equation}
\label{eq07122025}
\textstyle{\frac{1}{2}}(q-2\sqrt{q}-2)\le
|\bar{\Delta}\cap C|\le \textstyle{\frac{1}{2}}(\sqrt{q}+1)^2
\end{equation}
apart from four exceptions with $a_{12}=a_{22}=0$ and $a_{23}=1$, namely 
$$|\bar{\Delta}\cap C|=
\begin{cases}
\mbox{$q$, \quad \,\,\,for $\mathfrak{Tr}(a_{11})=0$, and   $a_{13}=a_{33}=0$;}\\
\mbox{$q-1$, for $\mathfrak{Tr}(a_{11})=0,$ and $a_{33}=a_{13}^2\ne 0$};\\
\mbox{$1$,\quad \,\,\,  for $\mathfrak{Tr}(a_{11})=1$,   and $a_{33}=a_{13}=0$};\\
\mbox{$0$,\quad \,\,\, for $\mathfrak{Tr}(a_{11})=1$,   and $a_{33}=a_{13}^2, a_{33}\ne 0$};\\ 
\end{cases}
$$
and two more exception with $a_{12}=a_{22}=a_{23}=0$ and $a_{11}a_{13}a_{33}\ne 0$, namely
$$|\bar{\Delta}\cap C|=
\begin{cases}
{\mbox{$q$, for $\mathfrak{Tr}(a_{11}a_{33}/a_{13}^2)=0$}};\\
{\mbox{$0$, for $\mathfrak{Tr}(a_{11}a_{33}/a_{13}^2)=1$}}.
\end{cases}
$$
\end{theorem}
\begin{proof}
 If $a_{12}=a_{22}=0, a_{23}=1$, the claim follows from Theorem \ref{pro19112025AA} and if $a_{12}=a_{22}=a_{23}=0$, the claim follows from Proposition \ref{pro02012026}. Otherwise, by Propositions \ref{pro12082025}, \ref{pro12082025A} and \ref{pro12082025AA}, $\mathcal{H}$ is irreducible, it has at least two points at infinity and they are non-singular. Therefore, results from the last part of Section \ref{av} apply. Since $N_q(\mathcal{G})=N_q(\mathcal{H})$,  (\ref{eq07122025})  follows from (\ref{eq19112025}), (\ref{eq20112025}), and (\ref{eq20112025A}).
 \end{proof}
The proof of Theorem \ref{th05112025} also shows the following result.
 \begin{corollary}
 \label{cor28112025}
 $$\textstyle{\frac{1}{2}}(q-2\sqrt{q}-2)\le |\Delta\cap C|\le
 \textstyle{\frac{1}{2}}(q+2\sqrt{q}-1)$$
 apart from two exceptions with $a_{12}=a_{22}=0, a_{23}=1$, namely
 $$|\Delta\cap C|=
 \begin{cases}
 q-1,\,\, \mbox{for $\mathfrak{Tr}(a_{11})=0$, and  $a_{13}=a_{33}=0$ or $a_{13}^2=a_{33}\ne 0$};\\
 0,\quad \,\,\,\,\,\, \mbox{for $\mathfrak{Tr}(a_{11})=1$, and  $a_{13}=a_{33}=0$ or $a_{13}^2=a_{33}\ne 0$};
 \end{cases} 
$$
and two more exceptions with $a_{12}=a_{22}=a_{23}=0$ and $a_{11}a_{13}a_{33}\ne 0$, namely
$$|\Delta\cap C|=
\begin{cases}
{\mbox{$q$, for $\mathfrak{Tr}(a_{11}a_{33}/a_{13}^2)=0$}};\\
{\mbox{$0$, for $\mathfrak{Tr}(a_{11}a_{33}/a_{13}^2)=1$}}.
\end{cases}
$$
\end{corollary}
\begin{remark}
\label{rem17112025} \emph{The Waterhous theorem may be used to compute the exact value of $N_q(\mathcal{H})$ and hence the size of $\bar{\Delta}\cap C$. However, the computation to be carried out seems to be feasible only for smaller values of $q$.}
\end{remark}

\section{Evaluation codes of symmetric functions on distinguished points of the affine plane}
\label{ap} Let $\Sigma$ and $\Pi$ be two copies of $AG(2,\mathbb{F}_q)$, up to isomorphisms, and equip $\Sigma$ and $\Pi$ with affine coordinates $(X_1,X_2)$ and $(Y_1,Y_2)$, respectively. Let $\pi$ be the quadratic map from $\Sigma$ to $\Pi$ defined by $(X_1,X_2)\rightarrow (Y_1,Y_2)$ where $Y_1=X_1+X_2$ and $Y_2=X_1X_2$. The non-distinguished points of $\Sigma$ are mapped by $\pi$ into the points on the line $Y_1=0$.

 \begin{lemma}\label{image}
     The images of the distinguished points of $\Sigma$ by the map $\pi$ form the set $\Delta=\bar{\Delta}\setminus\{(0,0)\}$ in $\Pi$, where $\bar{\Delta}$ is defined by (\ref{delta}).
 \end{lemma}
 \begin{proof}
Take a distinguished point $P=(x_1,x_2)$ in $\Sigma$. Let $\pi(P)=(y_1,y_2)$. Then $y_1\ne 0$. Since $y_1=x_1+x_2$ and $y_2=x_1x_2$, the roots of the polynomial $Z^2+y_1Z+y_2$ are $x_1$ and $x_2$. From this,  
     \begin{equation*}\label{eq}
             \left(\frac{x_2}{y_1}\right)^2+\frac{x_2}{y_1}+\frac{y_2}{y_1^2}=0.
         \end{equation*}
    As $x_2\in\mathbb{F}_q$, this yields $\mathfrak{Tr}(y_2/y_1^2)=0$. Let $a_0,\cdots,a_{\nicefrac{q}{2}-1}$ be the roots of the polynomial $\mathfrak{Tr}(Z)=Z^{\nicefrac{q}{2}}+\cdots+Z$ over $\mathbb{F}_q$. Then $\mathfrak{Tr}(Z)=(Z-a_0)\cdots(Z-a_{\nicefrac{q}{2}-1})$ whence
     \begin{equation}\label{tr}
         \mathfrak{Tr}\left(\frac{y_2}{y_1^2}\right)=\left(\frac{y_2}{y_1^2}-a_0\right)\cdots\left(\frac{y_2}{y_1^2}-a_{\nicefrac{q}{2}-1}\right)=\frac{(y_2-a_0y_1^2)\cdots(y_2-a_{\nicefrac{q}{2}-1}y_1^2)}{y_1^q}.
     \end{equation}
    Since $\mathfrak{Tr}(0)=0$, one of the coefficients $a_i$ vanishes, and we may assume $a_0=0$. Then (\ref{tr}) reads
    \begin{equation*}
        \frac{y_2(y_2-a_1y_1^2)\cdots(y_2-a_{\nicefrac{q}{2}-1}y_1^2)}{y_1^q}.
    \end{equation*}
It follows that the image of the distinguished points consists of the points which are on the line $Y_2=0$ or on the $\frac{1}{2}q-1$ parabolas of equation $Y_2=a_iY_1^2$ for $i=1,\cdots,\frac{1}{2}q-1$ but outside the line $Y_1=0$.
     \end{proof}

From the proof of Lemma \ref{image}, $\Delta$ comprises as many as $$(q-1)\left(\textstyle{\frac{1}{2}}q-1\right)+q-1=\textstyle{\frac{1}{2}}q(q-1)$$ points, in accordance to \cite{datta2023codes}.

 \subsection{Case of linear polynomials}
 \label{lp1}
Take the linear system $\Sigma_1=\Sigma_{2,1}$ over $\mathbb{K}$ consisting of all linear polynomials in $Y_1,Y_2$. According to Section \ref{smp},  $\Phi_2(\Sigma_1)$ is a linear system in the indeterminates $X_1,X_2$ such that $Y_1=X_1+X_2$ and $Y_2=X_1X_2$.
Evaluation of a symmetric polynomial $g\in \Phi_2(\Sigma_1)$ on the set of distinguished points in $\Sigma$ can be carried out by evaluating the corresponding polynomial $f$ with $g=\Phi_2(f)$ on the set $\Delta$ in $\Pi$. In particular, if $f$ has degree $1$, i.e. $f$ is represented by a line
$cY_2+mY_1+b=0$, then $g=\Phi_2(f)$ is represented by the hyperbole $cX_1X_2+m(X_1+X_2)+b=0$, and the weight $w(\ell)$ of the codeword in the Datta-Johnsen code $C_2'$ associated with $\ell$ is equal to $\frac{1}{2} q(q-1)-(\ell\cap \Delta)$. In Section \ref{linesec}, all possibilities for $\ell\cap \bar{\Delta}$ are determined. This gives the following results.
\begin{itemize}
    \item If $\ell$ has equation $Y_2=0$, then $|\ell\cap\bar{\Delta}|=q$. Thus, $|\ell\cap\Delta|=q-1$ by  $O\in\ell\cap\bar{\Delta}$ whence
    $$w(\ell)=\textstyle{\frac{1}{2}}q(q-1)-(q-1)=\textstyle{\frac{1}{2}}(q^2-3q+2)=\textstyle{\frac{1}{2}}(q-1)(q-2).$$
    \item If $\ell$ has equation $Y_2=mY_1+b, (m,b)\ne(0,0)$, then  $|\ell\cap\bar{\Delta}|=|\ell\cap\Delta|=\textstyle{\frac{1}{2}}(q-2)$ for $b\ne0$ while, if $b= 0, |\ell\cap\bar{\Delta}|=\textstyle{\frac{1}{2}}q-2$ and $|\ell\cap\Delta|=\textstyle{\frac{1}{2}}(q-2)$ by $O\in\ell\cap\bar{\Delta}$. Thus  $$w(\ell)=\textstyle{\frac{1}{2}}q(q-1)-\textstyle{\frac{1}{2}}(q-2)=\textstyle{\frac{1}{2}}(q^2-2q+2).$$
    \item If $\ell$ has equation $Y_1=b, b\ne0$, then  $|\ell\cap\Delta|=|\ell\cap\bar{\Delta}|=\textstyle{\frac{1}{2}}q$ whence
    $$w(\ell)=\textstyle{\frac{1}{2}}q(q-1)-\textstyle{\frac{1}{2}}q=\textstyle{\frac{1}{2}}q(q-2).$$
    \item If $\ell$ has equation $Y_1=0$, then $\ell\cap\bar{\Delta}=\{O\}$ and $\ell\cap\Delta=\emptyset$ whence
    $w(\ell)=\textstyle{\frac{1}{2}}q(q-1).$
\end{itemize}

Therefore, the weight distribution is $$\big\{\textstyle{\frac{1}{2}}(q-1)(q-2),\textstyle{\frac{1}{2}} q(q-2), \textstyle{\frac{1}{2}}(q^2-2q+2), \textstyle{\frac{1}{2}} q(q-1)\big\}.$$  Thus the reduced Datta-Johnsen code $C_2'$ is a $\left[\frac{1}{2} q(q-1),3,D\right]_q$ code which has minimum distance equal to $D=\frac{1}{2}(q-1)(q-2)$. Therefore, the Datta-Johnsen code has minimum distance $(q-1)(q-2)$ in accordance with \cite[Proposition 3.2]{datta2023codes} for the case $m=2$.

\subsection{Case of quadratic polynomials}\label{exa}
Let $\Sigma_2=\Sigma_{2,2}$ be the six-dimensional linear system over $\mathbb{K}$ consisting of all polynomials in $Y_1,Y_2$ of degree $\le 2$. The arising code has size $n=\frac{1}{2}q(q-1)$ and dimension $6$ whose weight distribution depends on the  possible intersections between $\Delta$ and a conic $\mathcal{D}$ in the affine plane $AG(2,\mathbb{F}_q)$. From Corollary \ref{cor28112025} the following result follows. 

\begin{theorem}
\label{th18032025} The reduced generalized Datta-Johnsen code $C_2(6)'$ arising from the linear system of all conics is  $\left[\frac{1}{2} q(q-1),6,\frac{1}{2}q(q-3)\right]_q$.
\end{theorem}
The linear system $\Sigma_2$ has linear subsystems $\Sigma_{2,2}(r)=\Sigma_2(r)$ for any degree $r$ for $1\le r \le 5$. Each of them gives rise to a generalized Datta-Johnsen code of size $\frac{1}{2}q(q-1)$ and dimension $r$ whose weight distribution and minimum distance depend on the intersections between the conics in $\Sigma_2(r)$ and $\Delta$. The possiblities for the sizes of such intersections are treated in Section \ref{conicsec}.
Here we work out two cases.

\subsubsection{Construction 1} In the first case, the linear system $\Sigma_2(3)$ of dimension $3$ will contain no reducible conic whose components are lines defined over $\mathbb{F}_q$. Actually, there exists just one such linear series;
see \cite{AL}. A useful description for our purpose is the odd characteristic version of the first example in \cite[Section 3.2]{GKNPS}. Look at the affine plane $AG(2,\mathbb{F}_{q^3})$ and its projective closure $PG(2,\mathbb{F}_{q^3})$ with homogeneous coordinates $(X:Y:Z)$. The projective  group $G=PGL(3,q)$ of $PG(2,\mathbb{F}_q)$ can be viewed as a subgroup of $PGL(3,q^3)$. The action of $G$ on $PG(2,\mathbb{F}_{q^3})$ produces three point-orbits, namely $PG(2,\mathbb{F}_q)$, the set of all points covered by lines of $PG(2,\mathbb{F}_q)$ and the set $\Lambda$ of the remaining
points. Here
\begin{equation*}
\label{eqA080125}
|\Lambda|=q^6+q^3+1-(q^2+q+1)-(q^2+q+1)(q^3-q)=q^6-q^5-q^4+q^3.
\end{equation*}
Choose a point $P=(a:b:c)\in\Lambda$ together with its Frobenius images $P_1=(a^q:b^q:c^q)$ and $P_2=(a^{q^2}:b^{q^2}:c^{q^2})$. They are the vertices of the triangle $PP_1P_2$ whose sides $\ell_1=PP_1$, $\ell_2=P_1P_2$ and $\ell_3=P_2P$ are disjoint from $PG(2,\mathbb{F}_q)$. If $\ell_i$ has equation $\ell_i(X,Y,Z)=a_iX+b_iY+c_iZ=0$ for $i=1,2,3$ then
$a_2=a_1^q,a_3=a_2^q,b_2=b_1^q,b_3=b_2^q,c_2=c_1^q,c_3=c_2^q$. This shows that the Frobenius image of $\ell_i$ is $\ell_{i+1}$ where the indices are taken$\pmod{3}$.

Let $\Sigma_2(3)$ be the net (linear system of projective dimension $2$) of $PG(2,\mathbb{F}_{q^3})$ comprising all conics through those points $P,P_1,P_2$. Clearly, $\Sigma_2(3)$ is generated by three reducible conics, namely those of equations $\ell_1(X,Y,Z)\ell_2(X,Y,Z)=0$,  $\ell_2(X,Y,Z)\ell_3(X,Y,Z)=0$ and $\ell_3(X,Y,Z)\ell_1(X,Y,Z)=0$, respectively, but   $\Sigma_2(3)$ also contains some further reducible conics of equations $\ell t=0$ where $\ell$ coincides with a side of the triangle, and $t$ is any line through the opposite vertex of $\ell$. The line $t$ may have at most one point of $PG(2,\mathbb{F}_q)$. Let $\bar{\Sigma}_2(3)$ be the set of all conics $\mathcal{C}_\lambda$ in $\Sigma_2(3)$ of equation $$\lambda\ell_1(x,y,z)\ell_2(x,y,z)+\lambda^q\ell_2(x,y,z)\ell_3(x,y,z)+\lambda^{q^2}\ell_3(x,y,z)\ell_1(x,y,z)=0$$ with $\lambda\in \mathbb{F}_{q^3}\setminus \{0\}$. The Frobenius image of $\mathcal{C}_\lambda$ is the conic of equation
$$\lambda^q\ell_2(x,y,z)\ell_3(x,y,z)+\lambda^{q^2}\ell_3(x,y,z)\ell_1(x,y,z)+\lambda^{q^3}\ell_1(x,y,z)\ell_2(x,y,z)=0.$$ Since $\lambda\in \mathbb{F}_{q^3}$, we have $\lambda^{q^3}=\lambda$, and hence $\bar{\Sigma}_2(3)$
consists of conics defined over $\mathbb{F}_q$, i.e. conics of $PG(2,\mathbb{F}_q)$. The total number of conics in $\bar{\Sigma}_2(3)$ equals $(q^3-1)/(q-1)=q^2+q+1$, i.e. the number of points of $PG(2,\mathbb{F}_q)$. It turns out that $\bar{\Sigma}_2(3)$ is a linear system of $PG(2,\mathbb{F}_q)$ of projective dimension $2$. Moreover, $\bar{\Sigma}_2(3)$ contains no reducible conic. This depends on the fact that any reducible conic in  $\Sigma_2(3)$ contains exactly one side of the triangle $PP_1P_2$ and hence it is not defined over $\mathbb{F}_q$ as the Frobenius map does not preserve any side of that triangle. Since $\bar{\Sigma}_2(3)$ has projective dimension $2$, there exists at least one conic $C\in \bar{\Sigma}_2(3)$ passing through the point $Y_\infty=(0:1:0)$ and tangent to the line $\ell_\infty$ of equation $Z=0$. In other words $C$ has homogeneous equation
\begin{equation}
\label{eq27112025}
U(X,Y,Z)=a_{11}X^2+a_{13}XZ+a_{23}YZ+a_{33}Z^2=0.
\end{equation}
In $\bar{\Sigma}_2(3)$, $C$ is the unique conic with equation of type (\ref{eq27112025}). In fact, if $V(X,Y,Z)=b_{11}X^2+b_{13}XZ+b_{23}YZ+b_{33}Z^2=0$ was the equation of another conic $D$ in $\bar{\Sigma}_2(3)$, then a $W(X,Y,Z)=b_{11}U(X,Y,Z)+a_{11}V(X,Y,Z)$ would be divisible by $Z$, and hence the conic of equation $W(X,Y,Z)=0$ would be a reducible conic in $\bar{\Sigma}_2(3)$.  Moreover, if $a_{11}$ in (\ref{eq27112025}) has zero trace, then we need to apply the change of coordinates $(X:Y:Z)\rightarrow (X':Y':Z')$ with $X'=\alpha X, Y'=Y, Z'=Z$ where $\mathfrak{Tr}(\alpha)=1$. In the new reference system, $C$ has equation of type (\ref{eq27112025}) with $\mathfrak{Tr}(a_{11})=1$.

Now, from Theorem \ref{pro19112025AA}, and Corollary  \ref{cor28112025}, the largest weight of the reduced generalized Datta-Johnson code $C_2(3)'$  does not exceed
$\frac{1}{2}(q+2\sqrt{q}-1)$. Therefore, it is a $\left[\frac{1}{2} q(q-1),3,D\right]_q$ linear code whose minimum distance $D$ is at least $\frac{1}{2} q(q-1)-\frac{1}{2}(q+2\sqrt{q}-1)=\frac{1}{2}\left(q^2-2q-2\sqrt{q}+1\right)$. Thus, the generalized Datta-Johnsen code $C_2(3)$ has minimum distance not smaller than $q^2-2q-2\sqrt{q}+1$. This shows Theorem \ref{main271224}.

\subsubsection{Construction 2}
In $AG(2,\mathbb{F}_q)$ with affine coordinates $(X,Y)$, let $\bar{\Sigma}_2(4)$ be the linear system of dimension $4$ consisting of all parabolas, that is, conics $C$ of equation
\begin{equation*}
   a_{11}X^2+a_{13}X+a_{23}Y+a_{33}=0.
\end{equation*}

Suppose first $a_{23}=0$. The following cases arise: $a_{11}=0$ and $C$ coincides with the vertical line $a_{13}X+a_{33}=0$;  $a_{11}\ne 0,a_{13}=0$ and $C$ coincides with the line $X+\sqrt{a_{33}/a_{11}}=0$;  or $a_{11}\ne 0,a_{13}\ne 0$, and $C$ splits into two vertical lines which are defined over $\mathbb{F}_q$, or $\mathbb{F}_{q^2}$ according as $\mathfrak{Tr}(a_{11}a_{33}/a_{13}^2)$ is equals $0$ or $1$. This together with Section \ref{lp1} show that if $a_{23}= 0$, then
$$ |\Delta\cap C|=
\begin{cases}
\textstyle{\frac{1}{2}}q,\,\,\quad {\mbox{for $a_{11}=0,a_{33}\ne 0$}};\\
0,\,\,\quad\,\,\,\, {\mbox{for $a_{11}=0,a_{33}=0$}};\\
\textstyle{\frac{1}{2}}q,\,\quad\, {\mbox{for $a_{11}\ne 0$, $a_{13}=0,a_{33}\ne 0$ ; }} \\
0,\,\,\quad\,\,\, {\mbox{for $a_{11}\ne 0$, $a_{13}=0,a_{33}=0$; }} \\
q,\,\,\quad\,\,\,\,{\mbox{for $a_{11}\ne 0, a_{13}\ne 0, a_{33}\ne 0$,  $\mathfrak{Tr}(a_{11}a_{33}/a_{13}^2)=0$; }} \\
\textstyle{\frac{1}{2}}q,\,\,\quad\,{\mbox{for $a_{11}\ne 0, a_{13}\ne 0, a_{33}=0$; }} \\
0,\,\,\quad\,\,\,\,\,{\mbox{for $a_{11}\ne 0, a_{13}\ne 0, a_{33}\ne0$,  $\mathfrak{Tr}(a_{11}a_{33}/a_{13}^2)=1$. }} \\
\end{cases}
$$

Suppose $a_{23}\ne 0$ and let $a_{23}=1$. From Theorem \ref{pro12082025AA}
$$ |\Delta\cap C|=
\begin{cases}
\textstyle{\frac{1}{2}}q-1,\,\,{\mbox{for $\mathfrak{Tr}(a_{11})=0$, $a_{33}\ne 0$  and $a_{33}\ne a_{13}^2$; }} \\
\textstyle{\frac{1}{2}}q,\,\,\,\,\,\,\,\,\,\,\,\, {\mbox{for $\mathfrak{Tr}(a_{11})=1$, $a_{33}\ne 0$  and $a_{33}\ne a_{13}^2$; }} \\
\textstyle{\frac{1}{2}}q-1,\,\,{\mbox{for $\mathfrak{Tr}(a_{11})=a_{33}=0$ and $a_{33}\ne a_{13}^2$; }} \\
\textstyle{\frac{1}{2}}q,\,\,\,\,\,\,\,\,\,\,\,\,{\mbox{for $\mathfrak{Tr}(a_{11})=1$, $a_{33}= 0$  and $a_{33}\ne a_{13}^2$; }} \\
q-1,\,\,\,\,\,\,{\mbox{for $\mathfrak{Tr}(a_{11})=0$,   and $a_{33}=a_{13}=0$;}} \\
0,\,\,\quad\,\,\,\,\,\,\,\,{\mbox{for $\mathfrak{Tr}(a_{11})=1$,   and $a_{33}=a_{13}=0$;}} \\
q-1,\,\,\,\,\,\,{\mbox{for $\mathfrak{Tr}(a_{11})=0$,   and $a_{33}=a_{13}^2, a_{33}\ne 0$;}} \\
0,\,\,\quad\,\,\,\,\,\,\,\,{\mbox{for $\mathfrak{Tr}(a_{11})=1$,   and $a_{33}=a_{13}^2, a_{33}\ne 0$.}}
\end{cases}
$$
Therefore, the linear system $\bar{\Sigma}_2(4)$ gives rise to a $\left[\frac{1}{2} q(q-1),4,D\right]_q$- code, whose minimum distance $D$ is $\frac{1}{2}q(q-3)$ and weight distribution is
$$\left\{\textstyle{\frac{1}{2}}q(q-3),\,\textstyle{\frac{1}{2}}(q^2-3q+2),\, \textstyle{\frac{1}{2}}q^2-q,\,\textstyle{\frac{1}{2}}q^2-q+1,\, \textstyle{\frac{1}{2}}q(q-1)\right\}.$$
This also shows that the generalized Datta-Johnsen code has minimum distance $q(q-3)$.

\section{Acknowledgments}
B. Gatti, G. Korchm\'aros and G. Schulte  have been partially supported by the Italian National Group for Algebraic and Geometric Structures and their Applications (GNSAGA - INdAM).


\bibliography{sn-bibliography}% common bib file
%\input sn-article.bbl

\end{document}